\documentclass[a4paper, 11pt, final,underruledhead,leqno]{paper}

\usepackage[mathscr]{eucal}
\usepackage{enumerate, amssymb, math, mathtools}

\begin{document}

\def\LineOn(#1,#2){\overline{{#1},{#2}\rule{0em}{1,5ex}}}
\def\collin{{\bf L}}
\def\Rad{\mathrm{Rad}}
\def\adjac{\mathrel{\sim}}
\def\nadjac{\mathrel{\not\adjac}}
\def\suport{\mathrm{supp}}
\def\msub{{\mathfrak y}}
\def\wsub{{\mathfrak w}}
\def\hipa{{\mathscr H}}
\def\hipy{{\mathcal H}}
\def\hipfun{{\goth h}}
\def\veblparal{\mathrel{\parallel^\circ}}

\def\lines{{\cal L}}
\def\planes{{\mathcal P}}
\def\bloki{{\mathscr B}}

\def\improper(#1){{#1}^{\infty}}
\def\pointnumbsymb{\mbox{\boldmath$v$}}
\def\pointnumb{\pointnumbsymb}
\def\linenumbsymb{\mbox{\boldmath$b$}}
\def\linenumb{\linenumbsymb}
\def\ranksymb{\mbox{\boldmath$r$}}
\def\rank{\ranksymb}
\let\pointrank\rank
\def\lineranksymb{\mbox{\boldmath$\kappa$}}
\def\linerank{\lineranksymb}
\def\dirranksymb{\mbox{\boldmath$\delta$}}
\def\dirrank{\dirranksymb}

\def\upcircled#1{{\bigcirc\mkern-13mu{\raisebox{-0.5pt}{\scriptsize$#1$}}}}
\def\gausnom#1#2#3{{\binom{#1}{#2}}_{#3}}
\def\topof{{\mathrm T}}
\def\starof{{\mathrm S}}
\def\inc{\mathrel{\,\rule[-3pt]{1pt}{12pt}\,}}
\def\ninc{\mathrel{\,\not\mkern1mu\rule[-3pt]{1pt}{12pt}\,}}

\def\Quadr{Q}

\def\afgeo{AG}
\def\projgeo{PG} 

\def\kolczaty{spiky} 
\def\luskwiaty{flappy}

\def\fixproj{\mbox{\boldmath$\goth P$}}
\def\fixaf{\mbox{\boldmath$\goth A$}}
\def\PencSpace(#1,#2){{\bf P}_{#1}({#2})}
\def\VerSpace(#1,#2){{\bf V}_{#1}({#2})}
\def\PolarSpace(#1){{\bf Q}({#1})}
\def\GrasSpace(#1,#2){{\bf G}_{#1}({#2})}
\def\GrasSpacex(#1,#2){{\bf G}^\ast_{#1}({#2})}

\def\myend{\hfill\mbox{\small$\bigcirc$}}


\def\plsX{partial linear space}
\def\pls{{\sf PLS}}
\def\apls{{\sf APLS}}
\def\aplsX{affine partial linear space}
\def\paplsX{partially affine partial linear space}

\newenvironment{cmath}{%
  \par
  \smallskip
  \centering
  $
}{%
  $
  \par
  \smallskip
  \csname @endpetrue\endcsname
}

\newenvironment{ctext}{%
  \par
  \smallskip
  \centering
}{%
 \par
 \smallskip
 \csname @endpetrue\endcsname
}

\newtheorem{lema}{Step}
		


\title{Hyperplanes, parallelism and related problems in Veronese spaces}
\author{K. Petelczyc, K. Pra\.zmowski, M. Pra\.zmowska, and M. \.Zynel}

\maketitle

\begin{abstract}
  We determine hyperplanes in Veronese spaces associated with projective spaces and polar spaces,
  and analyse the geometry of parallelisms induced by these hyperplanes.
  We also discuss if parallelisms on Veronese spaces associated with affine spaces can be imposed.
\end{abstract}

\begin{flushleft}\small
     \plsX, affine \plsX, hyperplane, correlation.
\end{flushleft}

\section*{Introduction}

The term `Veronese space' refers, primarily (and historically), to 
the structure of prisms in a projective space with  `double hyperplanes' as the 
points (see \cite{talin:vero}, \cite{talin:finiti}); 
after that it refers to an algebraic variety which represents this structure 
(cf.\ e.g.\ \cite{burau})
and as such it was generalized in recent decades, 
and its geometry was studied and developed 
(see e.g.\ \cite{ferrara}, \cite{vero-maldem1}, \cite{vero-maldem2}, \cite{vero-maldem3}).

A task to find a synthetic approach to (`original') Veronese spaces was undertaken
by the group of geometers around G. Tallini in the 70-th's 
and the results are presented in \cite{ferri}, \cite{mazzocca}, \cite{talin:vero}, \cite{talini}.
The way in which Veronese spaces were presented in those papers, where the points are two-element 
sets with repetitions with the elements being projective points, turned out to be 
fruitfully generalizable to an abstract construction, which associates with an arbitrary
partial linear space (or even more generally: with an arbitrary incidence structure)
another partial linear space (incidence structure resp.); this construction and its basic
properties were presented in \cite{verspac}. 
The point universe of the constructed Veronese space $\VerSpace(k,{\goth M})$ 
consists of the $k$-element sets with repetitions with the elements in the universe of
the underlying `starting' structure $\goth M$ 
(cf.\ definitions \eqref{def:verspace}, \eqref{def:verbloki}).

The construction discussed belongs to the family of, informally speaking, 
`multiplying' a given structure somehow in the spirit of `manifold theory':
the structure $\VerSpace(k,{\goth M})$ can be covered by a family of copies of $\goth M$
so as through each point of it there pass $k$ copies of $\goth M$. Any two distinct
copies of $\goth M$ in this covering do not share more than a single point.
The abstract schemas (in fact: partial linear spaces of particular type)
of the coverings induced by the construction of a Veronese space
were  studied in more detail in \cite{combver}.

In this paper we develop to some extent the theory of Veronese spaces associated with
spaces with some additional structures (like a parallelism), we discuss possibilities
to introduce parallelism in the Veronese spaces and related questions, all in a sense 
referring to a broad problem: ``hyperplanes and parallelisms".
The main approach is determining a hyperplane in a Veronese 
space associated with a linear space that contains hyperplanes (say: with a projective space)
and delete this hyperplane. The obtained reduct cannot be presented as a Veronese space
(\ref{thm:verprodNOTverred}). Nevertheless we could succeed in characterizing the hyperplanes in (`classical')
Veronese spaces associated with projective spaces (\ref{thm:hipaINprojver})
and prove that (similarly to the case when a hyperplane is deleted from a 
projective space)
from the obtained \paplsX\ the underlying Veronese space 
can be recovered (\ref{thm:verred2ver}).
We find also interesting examples of hyperplanes in Veronese spaces 
associated with polar spaces and we generalize our construction to this class also.
This completes the set of our main results.

In appendix we discuss another approach:
defining the Veronese space associated with a structure
with a parallelism (say: with an affine space).
There is no simple way to introduce a parallelism on the defined 
`Veronese product' (\ref{thm:noparINver}).
In both approaches we obtain a structure which can be covered by several copies of an affine space.
At the end of this paper we formulate several open problems.
They are out of the scope of our main reasoning but they seem closely related to it.


\section{Basics}

\subsection{Incidence structures, {\plsX s}}

In the paper we consider incidence structures i.e.\ structures
of the form 
${\goth W} = \struct{S,\bloki,\inc}$, 
where 
$\inc\;\subseteq\; S\times\bloki$ and $\bloki\neq\emptyset$. 

In the context of (general) incidence structures the binary joinability relation
(adjacency relation) 
$\adjac \,=\, \adjac_{\bloki}$ 
defined over an incidence structure $\struct{S,\bloki,\inc}$
plays a fundamental role:
\begin{equation*}
  a \adjac b \quad\iff\quad (\exists B\in\bloki)[a,b \inc B].
\end{equation*}
An incidence structure is {\em connected} when the relation 
(nonoriented graph) $\adjac$ is connected.
For blocks $L, M$ the formula $L\adjac M$ means that they intersect each other.

An incidence structure as above is {\em a  \plsX}\ (a \pls) if any two blocks
that are $\inc$-related to two distinct points coincide, and any block is on at least three points (most of the time, in 
the literature it is assumed that a block of a \pls\ has at least two points,
but {\em in this paper} we need a stronger requirement).
In that case the incidence structure 
${\goth W} = \struct{S,\bloki,\inc}$ is isomorphic to the structure
$\struct{S,\bloki^!,\in}$, where 
$\bloki^! = \big\{ \{ x\in S\colon x \inc B \}\colon B\in\bloki \big\}$;
this second approach to {\plsX s} will be preferred in the paper.

A {\em subspace} of a \plsX\ is a set $X$ of its points such that each line which meets $X$ in at least
two points is entirely contained in $X$. A subspace $X$ is {\em strong} 
(or {\em singular}, or {\em linear})
if any two of its points are adjacent.

A {\em linear space} is a \plsX\ in which any two points are adjacent.
Two classes of linear spaces are of primary interest in geometry: 
affine and projective
spaces, and structures associated with them are of primary interest in this paper.

\subsection{Veronese spaces}

Let $X$ be a nonempty set, we write $\msub_k(X)$ for the set of 
$k$-element sets with repetitions (multisets) with elements in $X$ i.e.\ the set  of
all the functions $f\colon X\longrightarrow \N$ such that
$|f| := \sum_{x\in X} f(x) = k$.
Then, directly by the definition, 
$f \in \msub_{|f|}(X)$.
If $f\in\msub_k(X)$ then
the set $\{ x\in X\colon f(x)\neq 0 \} =: \suport(f)$ is finite
and $\sum_{x\in X} f(x) = \sum_{x\in\suport(f)} f(x)$; in such a case
we write $f = \sum_{x\in\suport(f)} f(x)\cdot x$.
Finally, set
 $ \wsub_k(X) = \bigcup \left\{ \msub_l(X)\colon l < k \right\} $. 

Let ${\goth W} = \struct{S,\bloki,\inc}$ be an incidence structure and
$k$ be an integer.
We set
\begin{equation}\label{def:verbloki}
  \bloki^\upcircled{k} = 
  \big\{ \{e + rx \colon x\inc B\}\colon
  0 < r \leq k ,\; e\in\msub_{k-r}(S),\; B\in\bloki \big\}.
\end{equation}
Then we define (cf.\ \cite{verspac})
\begin{equation}\label{def:verspace}
  \VerSpace(k,{\goth W}) \quad:=\quad
  \struct{\msub_k(S),\bloki^\upcircled{k},\in}.
\end{equation}
The structure $\VerSpace(k,{\goth W})$ will be called the 
{\em ($k$-th) Veronese space associated with} $\goth W$ (or {\em over} $\goth W$).
We also say that it is a Veronese space {\em of level} $k$.

Let us quote after \cite{verspac}, \cite{veradjac}...\
the following simple observations.
\begin{fact}\label{fct:embed:gener}
  Let ${\goth W} = \struct{S,\bloki,\inc}$ be an incidence structure.
  Assume that $\goth W$ satisfies the extensionality principle:
  \begin{equation}
    B_1,B_2\in\bloki \land (\forall x)[x\inc B_1 \iff x\inc B_2]
    \implies B_1 = B_2.
  \tag{\text{\sf Ext}}
  \end{equation}
  %
  Let $k',r$ be nonnegative integers, $r\geq 1$. 
  Fix $e\in\msub_{k'}(S)$ and define the map
  \begin{eqnarray}\label{def:mu}
    \mu_r\colon	\msub_k(S) \ni f & \longmapsto & r f \in \msub_{r k}(S),
    \\ \label{def:tau}
    \tau_e\colon \msub_k(S) \ni f & \longmapsto & e +f \in\msub_{k+k'}(S)
  \end{eqnarray}
  Then $\mu_r$ is an embedding of $\VerSpace(k,{\goth W})$ into $\VerSpace({r k},{\goth W})$
  and $\tau_e$ is an embedding of $\VerSpace(k,{\goth W})$ into $\VerSpace({k'+k},{\goth W})$.
  \par
  Moreover, the identification 
  \begin{equation}\label{def:ast}
    {}^\ast\colon S\ni x \longmapsto 1 \cdot x (= x^\ast) \in\msub_1(S)
  \end{equation}
  is an isomorphism of
  $\goth W$ and $\VerSpace(1,{\goth W})$
\end{fact}
\begin{fact}\label{fct:combver:pls}
  The structure 
  \begin{equation}\label{def:combveron}
    \VerSpace(k,S) := 
    \VerSpace({k},{\struct{S,\{ S \},\in}})
  \end{equation}
  is a \plsX.
\end{fact}
The lines of $\VerSpace(k,S)$ will be called {\em leaves of} 
$\VerSpace(k,{\goth W})$.

The following is a folklore 
(see \cite{verspac}, \cite{veradjac})  
\begin{prop}
  A Veronese space associated with a {\plsX} is a \plsX.
\end{prop}
Recall after \cite[Prop. 2.9]{verspac} the following
\begin{fact}\label{fct:beztroj}
  Let $B$ be  a block of the Veronese space $\VerSpace(k,{\goth W})$
  associated with an incidence structure $\goth W$ and $\topof(B)$
  be the leaf of $\VerSpace(k,{\goth W})$ which contains $B$.
  If  a point $e$ is adjacent to at least three points on $B$
  then $e  \in \topof(B)$.
\end{fact}

Note that a leaf $e+ rS $ of a Veronese space can be identified with $e\in\wsub_k(S)$.
Indeed, 
a leaf of $\VerSpace(k,{\goth W})$ uniquely associated with $e$
has the form 
$ e + (k-|e|)S $.

Each leaf  $ e + (k-|e|)S $ of $\VerSpace(k,{\goth W})$
is a subspace of $\VerSpace(k,{\goth W})$ isomorphic to $\goth W$
under the map $x \longmapsto e + (k-|e|)\cdot x$.

\subsection{Veblen and Net Configurations in Veronese spaces}

A {\em Veblen configuration} consists of four lines no three concurrent, and
any two intersecting each the other. In a more explicit way we call 
{\em a Veblen Configuration} a 4-tuple of lines $L_1,L_2,M_1,M_2$ and a point $p$
such that $p\inc L_1,L_2$, $p\not\inc M_1,M_2$, $L_i\adjac M_j$ for $i,j=1,2$,
and $M_1\adjac M_2$. An {\em incomplete Veblen configuration} is the family as above, where
the condition $M_1\adjac M_2$ is not assumed.

A {\em Net configuration} consists of four lines $L_1,L_2,L_3,L_4$ and four points
$p_1,p_2,p_3,p_4$ such that
$p_i\inc L_i,L_{i+1 \mod 4}$, and
$p_i\not\adjac p_{i+2\mod 4}$
for $i=1,2,3,4$.
In plain words, it is a quadrangle in which  
diagonals do not exist.
In the literature the condition 
$L_i \not\adjac L_{i+2\mod 4}$
(= `opposite sides do not intersect') is frequently added;
in the configurations arising via affinizations such a 
requirement is too restrictive.

Together with these two configurations two configurational axioms are considered:
the {\em Veblen condition} and the {\em Net axiom}.
The Veblen condition states that every incomplete Veblen configuration closes 
i.e.\ if 
$p, L_1, L_2, M_1, M_2$ is an incomplete Veblen configuration defined above then
$M_1 \adjac M_2$.
The net axiom states that if $L_1,L_2,L_3,L_4$ is a quadrangle defined above,
$M_1 \adjac L_1,L_3$, $M_2 \adjac L_2,L_4$ then $M_1\adjac M_2$.
A \plsX\ which satisfies the Veblen condition is called {\em veblenian}.
Explicit forms of  the Veblen configurations contained in a Veronese space are shown in
\cite[Lem. 3.1]{verspac}. Let us recall this characterization.
\begin{fact}\label{fct:vebinver}
  Let\/ ${\goth M}_0$ be a \plsX\ and\/ ${\goth M} = \VerSpace(k,{\goth M}_0)$.
  A Veblen configuration contained in\/ $\goth M$ arises as a result of a natural 
  embedding of one of  the following figures
  \begin{enumerate}[\rm (i)]\itemsep-2pt
  \item\label{vebinver:typ1}
    a Veblen configuration contained in\/ 
    ${\goth M}_0$
  \item\label{vebinver:typ2}
    the set $\{ a + m \colon a \in A \}$ for a $4$-subset $A$ of a line $m$ of\/ ${\goth M}_0$,
  \item\label{vebinver:typ3}
    the set $\{ a + m \colon a \in A \} \cup \{ 2 m \}$ for a $3$-subset $A$ of a line $m$
    of\/ ${\goth M}_0$.
  \end{enumerate}
\end{fact}
From this one derives, in particular, that the Veblen axiom is not, generally, preserved
under Veronese products.
Let $\kappa({\goth M}_0) = \kappa$ 
be the size of a line of ${\goth M}_0$.
Note that $\VerSpace(k,{\goth M}_0)$ contains Veblen subconfiguration of 
the form \eqref{vebinver:typ2} iff $\kappa \geq 4$
and it contains Veblen subconfiguration of the form \eqref{vebinver:typ3}
iff $\kappa \geq 3$.
However, one particular case `behaves' more regularly.
With the classification of triangles in Veronese spaces given in \cite[Fact 4.1]{veradjac}
and standard computations we directly justify
\begin{fact}\label{fct:veb2-2ver}
  If a \plsX\ ${\goth M}_0$ is veblenian then 
  $\VerSpace(2,{\goth M}_0)$ is veblenian as well.
\end{fact}
Combining \ref{fct:vebinver} and \ref{fct:beztroj} we obtain
\begin{lem}
  Let $p, L_1, L_2, M_1, M_2$ form an incomplete Veblen configuration in a Veronese
  space associated with a \plsX\/ ${\goth M}_0$.
  Assume, moreover, that  points in one of the following  pairs 
  $(L_1\cap M_1,\,  L_2\cap M_2)$,
  $(L_1\cap M_2,\, L_2\cap M_1)$
  are collinear. 
  Then the lines $L_1,L_2,M_1,M_2$ are contained in a leaf.
  So, if\/ ${\goth M}_0$ is Veblenian, this configuration closes.
\end{lem}

Analogously, explicit forms of a realization of the Net Configuration in
a Veronese space were also established in \cite{verspac}.
It follows that the Net Axiom is not, generally, preserved under Veronese
`products'.
In the sequel we shall concentrate on a pretty special case of Veronese spaces,
namely on the structures $\VerSpace(k,{\goth M}_0)$ with $k=2$ 
(i.e.\ on those originally considered in the history).
In this case more `regular' figures appear.
Clearly, if $L_1,L_2,L_3,L_4$ form a quadrangle 
with $\topof(L_1) = \topof(L_3)$ then the quadrangle in question is obtained
by an embedding of a quadrangle in ${\goth M}_0$.
If ${\goth M}_0$ is a linear space (and this case is, primarily, studied
in this paper) and $\topof(L_1) = \topof(L_2)$ then such a quadrangle
has a diagonal.
Consequently, searching for a quadrangle without diagonals we can restrict ourselves
to the case 
$\neq (\topof(L_1),\topof(L_2),\topof(L_3),\topof(L_4))$.
Such a quadrangle without diagonals will be called {\em proper}
(note: this definition makes sense only in structures in which 
the notion of a `leaf' was introduced).
The following is just a matter of direct (though quite tedious) verification.
Let ${\goth M}_0$ be a linear space and  ${\goth M} = \VerSpace(2,{\goth M}_0)$
\begin{lem}\label{lem:vernet1}
  Let $L_1,K_1,L_2,K_2$ be lines of\/ $\goth M$ in 
  pairwise distinct leaves. These lines yield a quadrangle without diagonals
  iff one of the following holds, up to permutations: of the  pairs
  $(L_1,L_2),(K_1,K_2)$, and of lines in each of these pairs.
  \begin{sentences}\itemsep-2pt
  \item\label{vernet1:typ1}
    There are lines $m,n$ of\/ ${\goth M}_0$ and points
    $a_1,b_1 \in n$, $a_2,b_2\in m$ such that
    \begin{ctext}
      $L_1 = a_1 + m$, $L_2 = b_1 + m$,
      $K_1 = a_2 + n$, and $K_2  = b_2 + n$.
    \end{ctext}
  \item\label{vernet1:typ2}
    There are three lines $m,n,l$ of ${\goth M}_0$  and points $a,b,c$
    such that $a,b \in n$,  $a,c \in m$, $b,c \in l$ and
    \begin{ctext}
      $K_1 = 2n$, $K_2 = c+ n$,
      $L_1 = a + m$, and $L_2  = b + l$.
    \end{ctext}
  \end{sentences}
  The vertices of the respective quadrangles are
  $a_1 + a_2$, $a_1 + b_2$, $a_2 + b_1$, $b_1 + b_2$ in \eqref{vernet1:typ1},
  and
  $2a$, $a+c$, $2b$, $b+c$ in \eqref{vernet1:typ2}.
\end{lem}

\begin{lem}\label{lem:vernet2}
  Let $L_1,L_2$ be opposite sides of a proper quadrangle in
  ${\goth M}$.
  Let $K$ be a line of\/ $\goth M$ crossing $L_1,L_2$ 
  such that\/   $\topof(K)\neq \topof(L_1),\topof(L_2)$.
  Then one of the following holds.
  \begin{sentences}\itemsep-2pt
  \item\label{vernet2:typ1}
    $L_1 = a+ m$, $L_2 = b + m$ for $a\neq b$ and a line $m$ of\/
    ${\goth M}_0$; and
    $K = x + \LineOn(a,b)$ for some $x \in  m$ or
    $K = 2m$, when $a,b \in m$.
  \item\label{vernet2:typ2}
    $L_1 = 2n$, $L_2 = c+n$, $K = x + \LineOn(x,c)$
    for $x \in n$, $x \neq c$.
  \item\label{vernet2:typ3}
    $L_1 = a + m_1$, $L_2 = b + m_2$ for $a \neq b$ and 
    distinct lines $m_1,m_2$ of\/ ${\goth M}_0$ that share a point $c$.
    Then $K = 2 \LineOn(a,b)$ or $K = c + \LineOn(a,b)$.
  \end{sentences}
\end{lem}
Finally, gathering together the possibilities listed in 
\ref{lem:vernet1} and \ref{lem:vernet2}
we conclude with
\begin{prop}\label{prop:ver2netax}
  If\/ $K_1,K_2$ are two lines of\/ $\goth M$ which cross two other
  lines $L_1,L_2$ so as $L_1,K_1,L_2,K_2$ is a proper quadrangle,
  $L_3$ crosses $K_1,K_2$, and $K_3$ crosses $L_1,L_2$
  then $K_3,L_3$ share a point.
\end{prop}
Loosely (and not really strictly) speaking: 
$\VerSpace(2,{\goth M}_0)$ satisfies the Net Axiom.

\section{Hyperplanes in Veronese spaces}

Main problems of this paper concentrates around `affinizations' of Veronese structures.
These problems concern the question: {\em what are hyperplanes in Veronese structures, if
they exist}?

A set $X$ of points of a \plsX\ ${\goth M} = \struct{S, \lines}$ is {\em l-transversal}
if it meets every line of $\goth M$. 
Clearly, $S$ is l-transversal; a proper l-transversal subspace is
called a {\em hyperplane}.

A subspace $X$ of a partial linear space is called {\em \kolczaty}\ when 
through each point on $X$ there	goes a line that is not contained in $X$ (so, it meets $X$ 
in the given point only), and $X$ is {\em \luskwiaty}\ when through each line contained
in $X$ there passes a plane not contained in $X$.

A hyperplane $\hipa$ of a \plsX\ $\goth M$ determines a parallelism $\parallel_{\hipa}$ on the
lines not contained in $\hipa$ defined by the formula
$$
  L_1 \parallel_{\hipa} L_2 \iff L_1\cap\hipa = L_2\cap\hipa.
$$
Set 
  $\lines^\propto = \{ L\setminus\hipa\colon L\in\lines,\; L\nsubseteq\hipa\}$.
Recall that we have assumed $|L|\geq 3$ for every $L\in\lines$. Then each 
$l\in\lines^\propto$ uniquely determines $\overline{l}\in\lines$ such that
$l\subset\overline{l}$ and it makes sense to define the parallelism $\parallel_{\hipa}$
on $\lines^\propto$ by the condition
$l_1 \parallel_{\hipa}l_2 \iff \overline{l_1} \parallel_{\hipa}\overline{l_2}$.
Let us write
$$
  {\goth M}\setminus\hipa = \struct{S\setminus\hipa,\lines^\propto,\parallel_{\hipa}}.
$$
Then ${\goth M}\setminus\hipa$ is a \paplsX.
Note that it is not necessarily an \aplsX.
If $\hipa$ is \kolczaty\ then the points of $\hipa$ can be interpreted in terms of
${\goth M}\setminus\hipa$ as the equivalence classes of the parallelism $\parallel_{\hipa}$
(comp.\ definitions in \cite{segr2afin}).

Let ${\goth M}_0 = \struct{S,\lines_0}$ be a partial linear space;
let ${\goth M} = \VerSpace(k,{\goth M}_0)$ and
${\goth M}^\ast	= \VerSpace(k,S)$ be the structure of the leaves of $\goth M$.
Let us pass to our main goal: determine the hyperplanes in $\goth M$. 

The first solution that comes to mind, though natural, is defective:
\begin{rem}
  If\/ $\hipa_0$ is a hyperplane of\/ ${\goth M}_0$, then, clearly,
  the set $\msub_k(\hipa_0)$ is a subspace of\/ $\goth M$ but it is 
  not l-transversal for $k>1$.
\end{rem}
\begin{proof}
  Let $L\in\lines_0$ be not contained in $\hipa_0$, and
  $b\in L$ such that $b\notin\hipa_0$.
  Then $(b+L)\cap\msub_k(\hipa_0) = \emptyset$.
\end{proof}

Let $\hipa$ be a hyperplane of $\goth M$. Then, for each leaf 
$S' = e + (k-|e|)S$ the intersection $S'\cap\hipa$ is an l-transversal 
subspace of $S'$, so it is determined by an l-transversal subspace
$\hipa_e$ of ${\goth M}_0$. Write $\hipy$ for the set of all the hyperplanes
of ${\goth M}_0$. So, $\hipa$ determines 
via the formula
$\hipfun(e) = \{ x\in S\colon  e+(k-|e|)x \in \hipa\}$
a function
\begin{equation}\label{def:h}
  \hipfun\colon\wsub_k(S) \longrightarrow \hipy\cup \{S \}
\end{equation}
such that 
\begin{equation}\label{def:h2hipa}
  \hipa = \bigcup\{ e + (k-|e|) \hipfun_e\colon e \in\wsub_k(S) \}
  \quad(\text{we write } \hipfun_e = \hipfun(e)).
\end{equation}
Moreover, $\hipa$ is an l-transversal set in ${\goth M}^\ast$.
The following is a standard exercise.
\begin{lem}
  For every function as in \eqref{def:h} the set $\hipa$ defined by
  \eqref{def:h2hipa} is l-transversal in $\goth M$.
\end{lem}
%
In what follows we shall give a series of  interesting (we believe) 
examples of hyperplanes in  Veronese spaces associated with `classical'
geometries.

\subsection{First example: hyperplanes in Veronese spaces associated with
projective spaces}

At the beginning of this part let us recall properties characteristic for projective and
affine spaces. The  projective spaces are the veblenian linear spaces.
Affine spaces satisfy {\em the Tamaschke Bedingung}
({\em if a line parallel to one side of a triangle crosses a second side
then it crosses the third side as well}) and {\em Parallelogram Completion Condition}
({\em if  of two pairs of parallel lines three intersections of lines in pairs
of non-parallel lines exist, then the fourth intersection point exists as well}).

Let $\varkappa$ be a 
quasi-correlation in a projective space 
${\fixproj} = \struct{S,\lines}$ over a field with odd characteristic.
That means, there is a non-zero reflexive form $\xi$ on the vector space $\field V$
coordinatizing $\fixproj$ such that 
$\gen{u}\in\varkappa(\gen{v})$ is equivalent to $\xi(u,v) = 0$
for any non-zero vectors $u,v$ of $\field V$.
Consider ${\goth M} = \VerSpace(2,{\fixproj})$ 
(the primary example of a Veronese space: the Veronese variety, cf.\ \cite{talin:vero})
and define a function $\hipfun$
on $\wsub_2(S)$ as required in \eqref{def:h}.
Firstly, we set 
\begin{equation}\label{def:h-x}
  \hipfun(x) = \varkappa(x)\;\text{ for } x \in S.
\end{equation}
In accordance with \eqref{def:h}, either $\hipfun(0) = S$ or
$\hipfun(0) = h_0$ for some hyperplane $h_0$ of $\fixproj$.
\begin{lem}\label{lem:vari1}
  Suppose that $\hipfun(0) = h_0$ and assume that 
  $h_0 \nsubseteq\{ a\colon a\in\varkappa(a) \}$.
  Let\/ $\hipa$ be defined by \eqref{def:h2hipa}.
  Then $\hipa$ {\em is not} a subspace of\/ $\goth M$.
\end{lem}
\begin{proof}
  Let $a\in h_0$, $a\notin \varkappa(a)$, and let 
  $q \in \varkappa(a)\setminus h_0$. Then 
  $2a, a+q \in a + \LineOn(a,q)\cap\hipa$.
  Suppose $a + \LineOn(a,q) \subseteq\hipa$, then there exists a point 
  $q'\in\LineOn(a,q)$, $q\neq q'$ with $q'\in \varkappa(a)$, so (contradictory)
  $a \in \varkappa(a)$.
\end{proof}
Accordingly, to `extend' $\varkappa$ to a function $\hipfun$ that determines a 
hyperplane in $\goth M$ we must put
\begin{equation} \label{def:h-0}
  \hipfun(0 ) = S.
\end{equation}
%
Let $\hipa$ be defined by \eqref{def:h2hipa}.
\begin{lem}
  If $\varkappa$ is symplectic (all the points of\/ $\fixproj$  
  are selfconjugate), then\/ $\hipa$ determined by a function\/ $\hipfun$
  defined in \eqref{def:h-x}, \eqref{def:h-0} coincides with\/
  $\bigcup\{ x + \hipfun(x)\colon x \in S \}$.
\end{lem}
\begin{proof}
  Evident, as $x\in\varkappa(x)$ yields, in accordance with the definition,
  $2x = x+x\in\hipa$, for each point $x$ of $\fixproj$.
\end{proof}
  %
%
\begin{lem}\label{lema:symplect2hipa}
  If $\varkappa$ is symplectic, then\/
  $\hipa$  is a proper subspace of\/ $\goth M$     
  and therefore it is a hyperplane of\/ $\goth M$.
\end{lem}
\begin{proof}
  Let $L = a +L_0$ with $L_0\in\lines$; suppose
  $|L\cap \hipa|\geq 2$.
  Then there are at least two points $a+x_1,a+x_2$ in $L\cap\hipa$
  and either $a\neq x_1,x_2$ and thus $x_1,x_2\in\varkappa(a)$,
  which gives $L_0\subseteq\varkappa(a)$. 
  Consequently, $L\subseteq\hipa$.
  Or $a=x_1\neq x_2$ and then $x_2\in\varkappa(a)$;
  if $\varkappa$  symplectic then $x_1\in\varkappa(a)$ as
  well and the claim follows.
\end{proof}
Now, we are in a position to characterize all the hyperplanes in Veronese
spaces of level two associated with projective spaces.
\begin{thm}\label{thm:hipaINprojver}
  The set\/ $\hipa$ is a hyperplane in $\VerSpace(2,\fixproj)$ iff\/ $\hipa$ is defined by
  \eqref{def:h2hipa}, where\/ $\hipfun$ is defined by \eqref{def:h-x}, \eqref{def:h-0} for
  some symplectic quasi-correlation $\varkappa$ in\/ $\fixproj$.
\end{thm}
\begin{proof}
  The right-to-left implication follows directly from \ref{lema:symplect2hipa}.
  Let $\hipa$ be a hyperplane of $\VerSpace(2,\fixproj)$.
  Consider the binary relation $\perp$ on the points of $\fixproj$ defined by the condition
  $x \perp y \iff x+y \in \hipa$. Clearly, $\perp$ is symmetric. 
  Set $\hipfun(x) = \{ y\colon x \perp y \}$ for each point $x$ of $\fixproj$.
  By definition, 
  $x + \hipfun(x) = (x+S)\cap\hipa$ 
  is l-transversal in $x+S$, 
  so $\hipfun(x) = S$ or $\hipfun(x)$ is a hyperplane in $\fixproj$.
  As $\hipa$ is a proper subspace, for at least one $x$ we have $\hipfun(x)\neq S$.
  From \cite{muzal} we deduce that $\perp$ can be characterized by the formula
  $\gen{u} \perp \gen{v} \iff \xi(u,v) = 0$ for a sesquilinear form $\xi$
  defined on $\field V$. Let $\varkappa$ be the 
  quasi-correlation of $\fixproj$ determined by $\xi$.
  Note that $\msub_2(S) = \bigcup\{ x + S \colon x \in S \}$, so
  $\hipa = \bigcup \{ x + \hipfun(x) \colon x \in S \}$ i.e.\ 
  $\hipa$ is defined by \eqref{def:h2hipa} with $\hipfun$ satisfying \eqref{def:h-x}.
  Recall that $\hipfun(0) = \{ x\in S \colon 2x \in \hipa \}$
  From \ref{lem:vari1} we deduce (formally) that either $\hipfun(0)= S$ or 
  $\hipfun(0) = h_0$ is a hyperplane and this hyperplane is contained in the set of
  $\varkappa$-selfconjugate points.
  In both cases we conclude with $2S \subseteq \hipa$, which gives $a\perp a$ for each $a$.
  Consequently, $\varkappa$ is symplectic.
\end{proof}


So, from now on we assume that $\varkappa$ is symplectic.
Denote
  $$
    {\fixaf} := {\goth M}\setminus\hipa
  $$
the corresponding affine reduct.
From \ref{lema:symplect2hipa} it follows that $\fixaf$
is a \paplsX.
\begin{lem}
  The point set of\/ $\fixaf$ consists of all the multisets 
  $x+y$ with 
  $x,y\in S$ and $x\notin\varkappa(y)$ (so: $x\neq y$); 
  therefore this point set can be identified
  with a subset of $\sub_2(S)$.
\end{lem}
\begin{lem}\label{lema:kolczaty}
  The hyperplane $\hipa$ determined by a symplectic polarity is \kolczaty,
  but it is not \luskwiaty.
\end{lem}
\begin{proof}
  Take a point $e = x+y$ with $x \perp y$, $x\neq y$.	Then $y\in x^\perp$.
  There is a line $L_0$ through $y$ not contained in $x^\perp$.
  Set $L = x + L_0$, then $e \in L$ and $L$ not contained in $\hipa$.
  Indeed, suppose $x+z\in\hipa$ for $y\neq z \in L_0$; then $z \in x^\perp$,
  so $L_0\subseteq x^\perp$.
  \par
  Take any line $L_0$ of $\goth M$; then $L = 2L_0 \subseteq\hipa$.
  On the other hand any plane of $\goth M$ which contains $2L_0$
  is contained in $\topof(L) = 2S \subseteq\hipa$ which yields
  our second claim.
\end{proof}
It is rather easy to observe the following
\begin{fact}\label{fct:strongINreduct}
  The maximal strong subspaces of\/ $\goth M$ are the leaves of\/ $\goth M$.
  Consequently, the maximal strong subspaces of\/ $\fixaf$
  are affine spaces of the form $(x+S)\setminus\varkappa(x)$, $x\in S$
  (cf.\ {\upshape\cite[Fact 2.2]{veradjac}, \cite[Prop. 2.11]{verspac}}).
\end{fact}
As a direct consequence of \ref{fct:strongINreduct}
and \ref{fct:veb2-2ver} (cf.\ \cite[Lem. 2.4]{segr2afin})
we obtain
\begin{lem}\label{lem:vebparallel}
  \begin{sentences}\itemsep-2pt
  \item\label{vebparal:1}
    Each  Veblen subconfiguration `with diagonals' 
    (i.e each projective quadrangle) of $\goth M$
    is contained in a leaf.
  \item\label{vebparal:2}
    The relation of Veblen parallelism is properly definable in\/ $\fixaf$
    by the formula \eqref{def:veblparal}, i.e.\ lines 
    are in the relation $\veblparal$ when they are on a (affine) plane and are 
    parallel on that plane.
  \end{sentences}
\end{lem}
As an important by-product of \ref{fct:veb2-2ver}
we get  nearly immediately
\begin{prop}
  The structure $\fixaf$ satisfies the Tamaschke Bedingung and 
  the Parallelogram Completion Condition. 
\end{prop}

The framework proposed admits some degenerations. Namely, we cannot expect that $\varkappa$
is nondegenerate, and thus it may happen that $\hipa \supseteq 2S$ contains some leaves of 
the form $x+S$ as well. And then also lines contained in these leaves cannot 
be extended to ``proper" planes of $\fixaf$ and $\hipa$ is `more non-\luskwiaty\ than
expected'. In essence, this happens in every odd dimension of $\field V$.

In what follows we assume that 
{\em $\varkappa$ is nondegenerate}, and then $\dim({\field V})$ is even.
Let us examine the structure of the parallelism $\parallel_{\hipa}$ and of the
horizon determined by it. We begin with an evident observation
\begin{lem}\label{lem:2types-of-dir}
    Let $e\in\hipa$. Set 
    ${\fixaf} = {\goth M}\setminus\hipa$.
    Then one of the following two possibilities holds
    \begin{itemize}\def\labelitemi{\textbullet}\itemsep-2pt
    \item
      {\em For any two lines $L_1,L_2$ of $\fixaf$ which pass through $e$
      there is a plane of $\fixaf$ which contains them; in consequence, 
      $L_1\veblparal L_2$}: 
      this happens when $e = 2x$ with $x\in S$.
      In this case the two leaves of $\goth M$ through $e$ are $x+S$ and $2S\subseteq \hipa$
      and thus $L_1,L_2$ are determined by two lines, both two in $x+S$.
    \item
      {\em There are two lines $L_1,L_2$ through $e$ such that 
      $L_1 \not\veblparal L_2$ 
      and any $L_3$ through $e$
      (i.e.\ any $L_3$ with $L_3\parallel L_1$) satisfies
      $L_3\veblparal L_1$ or $L_3\veblparal L_2$}: 
      this happens when $e = x+y$ with $x\neq y$, we take 
      $L_1 \subseteq x+ S$, $L_2\subseteq y+S$ (comp. \ref{lema:kolczaty}).
    \end{itemize}
\end{lem}
This allows us to distinguish two types of directions in $\fixaf$;
directions of the first type 
  ($[L]_\parallel$ when $L\parallel L'$ iff $L\veblparal L'$ for each line $L'$)
  correspond to the elements of one totally deleted leaf $2S$.
In any case, a direction of $\fixaf$ uniquely corresponds to a point in $\hipa$;
in what follows we shall frequently identify corresponding two objects.

For a direction $\goth a$ (an equivalence class under {\em a parallelism $\parallel'$}) 
and a set of points $X$ write
  \begin{ctext}
    ${\goth a}\inc X$ (in words: ${\goth a}$ is incident with $X$)
    when ${\goth a} = [L]_{\parallel'}$ for a line 
    $L \subseteq X$.
  \end{ctext}
\begin{rem}\normalfont
  Note that though the two relations $\parallel$ and $\veblparal$ {\em do not
  coincide}, the relation $\veblparal$ 
  {\em is	a partial parallelism} so, it also determines its directions (equivalence classes).
  And, formally speaking, a point $x+y\in\hipa\setminus 2S$ determines
  {\em two distinct} $\veblparal$-directions.

  Let us stress on the fact that the distinction formulated in \ref{lem:2types-of-dir}
  refers entirely to $\parallel_{\hipa}$-directions; loosely speaking ${\goth a}\in 2S$ iff
  it is incident with exactly one leaf, and other $\parallel$-directions are incident with two leaves.
  But {\em each $\veblparal$-direction is incident with exactly one leaf}!
\myend
\end{rem}

$\hipa$ is \kolczaty, so its points can be identified with the equivalence classes of 
$\parallel_{\hipa}$. However, $\hipa$ is not \luskwiaty, so lines on $\hipa$ cannot
be identified, in general, with directions of planes of $\fixaf$ and the standard way 
to recover $\goth M$ from $\fixaf$ fails.
This recovering is still possible though, only the recovering procedure must
be complicated a bit.
\begin{lem}\label{lem:vebparallel:1}
    The lines of the horizon of\/ $\fixaf$
    that are contained in a leaf of the form $x+S$ are definable in\/ $\fixaf$.
\end{lem}
\begin{proof}
  Firstly, we note that the class of planes of $\fixaf$ can be defined in $\fixaf$.
  Indeed, let 
  $\Delta$ be a triangle with the sides $L_1,L_2,L_3$ and the vertices $e_1,e_2,e_3$,
  $e_i\not\inc L_i$ for $i=1,2,3$, such that 
  $e_1 \adjac e_0 \inc L_1$ for some $e_0\neq e_2,e_3$. 
  Then the set
  %
  \begin{equation}\label{def:hip2plane}
    \pi(L_1,L_2,L_3) \quad := \quad \bigcup\left\{  
    L\colon L\parallel L_1\land L\adjac L_2,L_3
    \right\}
  \end{equation}
  is a plane in $\fixaf$ i.e.\ 
  $\pi(L_1,L_2,L_3) = a + A$ for 
  a point $a$ of $\fixproj$ and a plane $A$ of the affine reduct $\fixproj\setminus\hipfun(a)$,
  and each plane of $\fixaf$ has a form as in \eqref{def:hip2plane}.
  %
  %
  So, let $\planes$ be the class of planes. The collinearity of the required form
  is defined by
  \begin{equation}
    \collin([L_1]_\parallel,[L_2]_\parallel,[L_3]_\parallel)
    \iff (\exists A\in\planes)(\exists L'_1,L'_2,L'_3\subseteq A)\;
    [\wedge_{i=1}^3 L'_i\parallel L_i].
  \end{equation}
  This argument closes the reasoning.
\end{proof}
\begin{rem}\normalfont
  In the ordinary affine geometry the formula \eqref{def:hip2plane}
  defines a plane for every triangle $L_1,L_2,L_3$.
  In the case of  Veronese spaces we must be cautious.
  Indeed, \em 
  if lines $L_0,L_1,L_2,L_3$ yield in $\goth M$ a Veblen figure of 
  the form \ref{fct:vebinver}\eqref{vebinver:typ2} or \ref{fct:vebinver}\eqref{vebinver:typ3}
  then $\pi(L_1,L_2,L_3) = \msub_2(m)$ and, clearly, the latter is not a plane.
\myend
\end{rem}

Though the lines on $2S$ are not ``improper lines" of affine planes, we can
recover also these lines in terms of $\fixaf$.

\begin{lem}\label{lem:hipa2-2lines}
  The lines of the horizon of\/ $\fixaf$ that are contained in the leaf $2S$
  are definable in terms of\/ $\fixaf$.
\end{lem}
\begin{proof}
  Let $\cal S$ stand for the class of the maximal strong subspaces of $\fixaf$;
  it is definable. From \ref{fct:strongINreduct}
  ${\cal S} = \{a + (S \setminus \varkappa(a)) \colon a \in S\}$. 
  %
  So, with each line $L$ of $\fixaf$ we have a definable set $\topof(L)$ with
  $L\subseteq\topof(L)\in{\cal S}$.
  In particular, the notion of a proper quadrangle can be expressed in terms of $\fixaf$.
  We have for ${\goth a}_1,{\goth a}_2,{\goth a}_3\in 2S$
  \begin{multline}
    \collin({\goth a}_1,{\goth a}_2,{\goth a}_3) \iff
     \exists L_1,L_2,L_3 \Big(
      \exists L',L'',M',M'' \text{: - a proper quadrangle in } {\fixaf}
  \\
     \big(L_1,L_2,L_3 \adjac L',L'' \Land \land_{i=1}^3{\goth a}_i\inc \topof(L_i)\big)\Big).
  \end{multline}
  The claim is immediate after \ref{lem:vernet1} and \ref{lem:vernet2}.
\end{proof}
As an important consequence we obtain now
\begin{thm}\label{thm:verred2ver}
  The underlying Veronese space\/ $\goth M$ can be recovered from its
  affine reduct\/ $\fixaf$.
\end{thm}

\begin{note}\normalfont
  Slightly rephrasing the proof of \ref{lem:hipa2-2lines} with the help of
  \ref{prop:ver2netax} one can prove that 
  \em
  the parallelism $\parallel_{\hipa}$
  can be defined in terms of the incidence structure of $\fixaf$.
  \normalfont
  Indeed, two lines are parallel either when they are in one leaf: then their
  parallelism coincides with $\veblparal$. Or they are in distinct leaves of $\goth M$:
  then they can be completed to a proper net; if they do not intersect each other, their common point must 
  lie in $\hipa$.
\end{note}

As we already noted, the leaves of $\fixaf$ carry the structure of 
an affine space.   
Particularly, if $\fixproj = \projgeo(n,q)$ ($2 {\not|} n,q$)
then the leaves of $\fixaf$ have the structure of $\afgeo(n,q)$.
But the geometry of affine reducts of Veronese spaces associated with
projective spaces and the geometry of Veronese spaces associated with 
affine spaces are essentially distinct.

\begin{thm}\label{thm:verprodNOTverred}
  Let\/ $\hipa$ be a hyperplane in\/ $\VerSpace(2,\fixproj)$.
  Then 
  $$
    \VerSpace(2,\fixproj)\setminus\hipa \not\cong
    \VerSpace(k,{\goth A}_0)
  $$
  for every affine space ${\goth A}_0$ and every integer $k$.
\end{thm}
\begin{proof}
  Let $x+S$ be a leaf of 
  ${\goth M} = \VerSpace(2,\fixproj)$. Then the corresponding leaf
  of ${\fixaf} = {\goth M}\setminus\hipa$ is $S':= (x+S)\setminus\hipa$; 
  let   ${\goth A}'_0$ be the restriction of 
  ${\goth A}$ to $S'$. 
  Each leaf of $\fixaf$ is isomorphic to ${\goth A}'_0$.
  And there pass exactly two leaves through each point of $\fixaf$.
  Suppose an isomorphism exists. Then $k=2$ and ${\goth A}_0 = {\goth A}'_0$.
  To close the proof it suffices to observe that the Veronese space 
  $\VerSpace(2,{\goth A}_0)$ associated with an affine space satisfies
  the Net Condition.
  The reduct $\fixaf$ does not satisfy this condition:    
  two lines through $x+y\in\hipa$, one in $x+S$ and the second in $y+S$
  can be completed to a net on its proper points, and clearly they do not 
  cross each other in $\fixaf$.
\end{proof}

\subsection{Generalization: hyperplanes in Veronesians with level $k>2$
associated with projective spaces}

The construction of a hyperplane determined by a symplectic form can be
applied to Veronesians of the level greater than 2 as well.
Let $\eta\colon V^k\longrightarrow F$ be a $k$-linear nondegenerate alternating
form, defined on a vector space $\field V$ with $V$ being its set of vectors, 
let $F$ be the set of scalars.
Recall two basic properties of $\eta$:
\begin{enumerate}[a)]\itemsep-2pt
\item\label{eta:symetria}
  $\eta(v_1,\ldots,v_k) = 0$ yields $\eta(v_{\sigma(1)},\dots,v_{\sigma(k)}) = 0$
  for every permutation $\sigma$ of $\{ 1,\dots, k \}$.
\item\label{eta:powtorki}
  if $v_i = v_j$ for $1\leq i<j\leq k$ then $\eta(v_1,\dots, v_k) = 0$.
\end{enumerate}

For a family $q_1 = \gen{v_1},\dots, q_k = \gen{v_k}$ of points of 
$\fixproj = \PencSpace(1,{\field V}) = \struct{S,\lines}$ we write 
$\perp_\eta(q_1,\dots,q_k)$ when $\eta(v_1,\dots,v_k) = 0$.
From the property \ref{eta:symetria}) of $\eta$ we get that the relation $\perp = \perp_\eta$ is
fully symmetric.

Define
\begin{equation}\label{def:eta2hipa}
  \hipa \colon = \left\{ q_1 + q_2 + \ldots + q_k \colon
  q_1,\dots, q_k \text{ - points of } \fixproj,\; \perp(q_1,\dots, q_k) \right\}
\end{equation}
\begin{fact}\label{fct:eta2hipa}
  The set\/ $\hipa$ defined by \eqref{def:eta2hipa} is a (nondegenerate) hyperplane
  in\/ $\VerSpace(k,\fixproj)$.
\end{fact}

From the property \ref{eta:powtorki}) of $\eta$, the points of $\fixaf =\VerSpace(k,\fixproj)\setminus\hipa$ i.e.\ the set $\msub_k(S)\setminus\hipa$ are $k$-subsets of $S$ (that we denote by $\sub_k(S)$) and therefore the affine reduct
$\fixaf$ can be 
characterized as follows:

%
\begin{description}\itemsep-2pt
\item[{\normalfont the points of $\fixaf$:}]
  $\left\{ \{q_1,\ldots,q_k\}\in\sub_k(S)\colon \not\perp(q_1,\dots, q_k) \right\}$,
\item[{\normalfont the lines of $\fixaf$:}]
  $\big\{ \{ \{ q_1,\dots, q_{k-1},x \}\colon x \in L, \not\perp(q_1,\dots, q_{k-1},x)\}
  \colon L\in\lines, L\nsubseteq (q_1,\dots, q_{k-1})^\perp \big\}$.
\end{description}

The structures obtained are more `Grassmannians' than `Veronesians':
they are defined on sets {\em without repetitions}! 
The problem to enter deeper into geometry of such reducts of Grassmannians is
certanly interesting, but it is not the goal of this paper.

\subsection{More sophisticated example: Veronese  spaces associated
with polar spaces}

Let us begin with three rather evident observations

\begin{fact}\label{fct:sub2hipa}
  Let\/ $\hipa$ be a hyperplane in a partial linear space\/
  ${\goth N} = \struct{S,\lines}$ and let\/ $S_0\subseteq S$.
  Set\/
  $\lines[S_0] = \{ L\in\lines\colon L\subseteq S_0 \}$.
  If\/  $\lines_0\subseteq\lines[S_0]$, then\/
  $\hipa\cap S_0$ is an l-transversal subset in 
  $\struct{S_0,\lines_0}$.
\end{fact}
In the notation of the fact \ref{fct:sub2hipa} we write
${\goth N}[S_0] = \struct{S_0,\lines[S_0]}$. 
Generally, ${\goth N}[S_0]$ needs not to be a partial linear space,
but for some `nonsense-reasons' only: its line set may be empty, it may
have isolated points etc.

\begin{fact}\label{fct:ver4sub:points}
  Let\/ ${\goth M}_0 =\struct{S_0,\lines_0}$ be a partial linear space
  and let\/ $S'_0\subseteq S_0$. 
  Then\/
  $\VerSpace(k,\struct{S'_0,\lines[S'_0]})$ and\/
  $\VerSpace(k,\struct{S_0,\lines_0})[\msub_k(S'_0)]$ coincide.
\end{fact}
\begin{proof}
  The point sets of both structures are equal: just from definitions.
  \par
  Let $L$ be a line of $\VerSpace(k,\struct{S'_0,\lines[S'_0]})$.
  So, $L = e + (k-|e|)L_0$, where: 
  $L_0\in\lines[S'_0]\subseteq\lines_0$, so  $L_0\subseteq S'_0$,
  and $e\in\msub_k(S'_0)\subseteq\msub_k(S_0)$.
  Finally, $L$ is a line of $\VerSpace(k,\struct{S_0,\lines_0})$
  and $L\subseteq\msub_k(S'_0)$.
  Conversely, let $L = e+(k-|e|)L_0$ be a line of $\VerSpace(k,\struct{S_0,\lines_0})$.
  Assume that $L\subseteq\msub_k(S'_0)$. Then $e\in\msub_k(S'_0)$ and
  $L_0\subseteq S'_0$, $L_0\in\lines_0$ and thus $L$ is a line of 
  $\VerSpace(k,\struct{S'_0,\lines[S'_0]})$.
\end{proof}

\begin{fact}\label{fct:ver4sub:lines}
  If\/ ${\goth M}' = \struct{S,\lines'}$ and\/ ${\goth M}'' = \struct{S,\lines''}$ are partial linear 
  spaces such that\/ $\lines''\subseteq\lines'$ then the line set of\/
  $\VerSpace(k,{\goth M}'')$ is a subset of the line set of\/ $\VerSpace(k,{\goth M}')$.
\end{fact}

Let us quote the standard models of polar spaces (in what follows 
{\em a polar space} will always mean one of the below).
Let $\varpi$ be a polarity in $\fixproj = \struct{S,\lines}$; 
let 
$\Quadr_0(\varpi) = \{ p\colon p\in\varpi(p) \}$
be the set of points of $\fixproj$ that are self-conjugate under $\varpi$
and 
$\Quadr_1(\varpi) = \{ L\colon L\subseteq\varpi(L) \}$
be the set of selfconjugate lines.
Then the polar space determined by $\varpi$ is the structure 
$\PolarSpace(\varpi) := \struct{\Quadr_0(\varpi),\Quadr_1(\varpi)}$
provided $\varpi$ is symplectic or $\varpi$ is quadratic with index at least 2
i.e.\ $\Quadr_2(\varpi)\neq \emptyset$.
In corresponding cases we have 
$\PolarSpace(\varpi) = \struct{S,\Quadr_1(\varpi)}$ 
and
$\PolarSpace(\varpi) = \fixproj[\Quadr_0(\varpi)]$.
%
%
Let us recall also that an {\em affine polar space} is an affine reduct of 
a polar space obtained by deleting a hyperplane of it. In our approach this 
hyperplane can be always considered as the restriction to $\Quadr_0$ of a hyperplane
of \fixproj.

The following will be needed, that is known in the literature.

\begin{fact}\label{fct:mocne:polar}
  A maximal strong subspace of a polar space  (of an affine polar space) is
  a projective (affine, resp.) subspace of\/ \fixproj\ (of the affine reduct of\/ \fixproj).
  \par
  Any two planes of a polar space and of an affine polar space can be joined 
  by a sequence of planes where each two consecutive planes share a line 
  (polar spaces and affine polar spaces are {\em strongly connected}, cf.\  {\upshape\cite{naumo}}).
\end{fact}

Let us consider the structure 
${\goth M} := \VerSpace(k,{\PolarSpace(\varpi)})$.
Then as an immediate consequence of \ref{fct:ver4sub:lines}
and \ref{fct:ver4sub:points} resp. we have
\begin{lem}\label{lem:proj2polar}\strut
\begin{sentences}\itemsep-2pt
\item 
  When $\varpi$ is symplectic, 
  the line set of 
  $\goth M$ is a subset of the line set of\/ $\VerSpace(k,\fixproj)$,
\item
  ${\goth M} = \VerSpace(k,\fixproj)[\msub_k(\Quadr_0(\varpi))]$
  when $\varpi$ is quadratic.
\end{sentences}
\end{lem}
Let $\hipa$ be a hyperplane of $\VerSpace(k,\fixproj)$
determined by a polarity $\varkappa$.  
From \ref{fct:sub2hipa} and \ref{lem:proj2polar} we obtain
\begin{fact}\label{fact:hipyiverpolar}
  The set 
  $\hipa\cap\msub_k(\Quadr_0(\varpi))$ 
  is a hyperplane in\/ $\goth M$.
\end{fact}

We have no proof, but it seems that

\begin{conj}
  Those of the form in {\upshape\ref{fact:hipyiverpolar}} are the only 
  possible hyperplanes in $\goth M$.
\end{conj}

In the sequel we write simply ${\goth M}\setminus\hipa$ instead of
${\goth M}\setminus(\hipa\cap\msub_k(\Quadr_0(\varpi)))$.

\begin{lem}
  The leaves of\/ $\goth M$ and their restrictions in 
  ${\goth M}\setminus\hipa$
  are definable in the internal geometry of $\goth M$
  (of ${\goth M}\setminus\hipa$, resp.).
\end{lem}
\begin{proof}
  In the first step we consider strong subspaces of 
  $\goth M$ and ${\goth M}\setminus\hipa$. Clearly, they have form
  $e+(k-|e|)X$ ($(e+(k-|e|)X)\setminus\hipa = e+(k-|e|)A$, resp.), where
  $X$ is a strong subspace of $\PolarSpace(\varpi)$
  ($A$ is a strong subspace of the affine polar space 
  $\PolarSpace(\varpi)\setminus\varkappa(e)$ resp.).
  \par
  Therefore, in the second step we can define the class $\planes$ of planes of
  $\goth M$ (of ${\goth M}\setminus\hipa$).
  \par
  In the third step we define a point-to-point relation $\gamma$:
  $a\mathrel{\gamma}b$ when $a$ and $b$ can be joined by a sequence of elements
  of $\planes$ where each two consecutive elements share a line. 
  \par
  In the last step we note that the leaves in question are the
  equivalence classes under the relation $\gamma$.
\end{proof}
Let $\varkappa$ be symplectic and $k=2$. 
Applying the techniques of the previous subsection we can prove now
\begin{thm}
  The underlying Veronese space $\goth M$ can be defined in terms of
  its affine reduct ${\goth M}\setminus\hipa$.
\end{thm}

\section{Appendix: multiplying a parallelism} 

In this additional part we shall discuss the following problem:
{\em how to extend a parallelism from a given structure to a Veronese 
space associated with it}?

Let $\struct{S_0,\lines_0}$ be a \plsX. 
A relation $\parallel\,\subseteq\,\lines_0\times\lines_0$ is called
a {\em partial parallelism} ({\em a preparallelism}) if 
it is an equivalence relation such that
distinct parallel lines are disjoint.
If $\parallel$ is a preparallelism as above then the structure 
$\struct{S_0,\lines_0,\parallel}$ is called {\em a \paplsX};
the relation $\parallel$ is called {\em a parallelism} and the
respective structure is called {\em an \aplsX}\
when the following form of the (affine) Euclid axiom holds:
{\em the equivalence class $[L]_\parallel$} (the {\em direction})
{\em of each line $L$ covers $S_0$}.
The most celebrated class of {\aplsX s} constitute affine spaces.
Recall that the natural parallelism of an affine space coincides with the relation
(so-called {\em Veblen parallelism}) $\veblparal$ defined by the formula
\begin{multline}\label{def:veblparal}
  L_1 \veblparal L_2 \quad \iff \quad L_1 = L_2  \Lor \big( L_1 \not\adjac L_2  
  \\
  \&\; \text{ there are two lines } L', L'' 
  \text{ through a point } p 
  \\
  \text { such that } 
  p \ninc L_1, L_2 \Land L_1,L_2 \adjac L',L' 
  \\
  \& \text{ there are collinear points }
  a_1 \in L_1 \cap L', \; a_2 \in L_2 \cap L'' \big).
\end{multline}
Intuitively speaking: $L_1\veblparal L_2$ when $L_1,L_2$ are on a plane and either
coincide or have no common point.

Most of the time (e.g.\ in the case of affine spaces and their Segre products,
cf.\ \cite{segr2afin}) \eqref{def:veblparal} is equivalent to a simpler formula 
with the condition
{\em there are collinear points
  $a_1 \in L_1 \cap L', \; a_2 \in L_2 \cap L''$}
on its right-hand side omitted.
Here, we must handle with this more complex formula since
in a Veronese space not every triangle determines a plane.

\medskip

Let us start with the following approach, which may seem, at the first look, natural
(especially taking into account similarities between Veronese and Segre products).
Let ${\goth A}_0 = \struct{S_0,\lines_0,\parallel_0}$ be a partial affine
partial linear space. We define on the lines of 
$\VerSpace(k,{\struct{S_0,\lines_0}})$
the relation $\parallel$ by the formula
\begin{multline}\label{def:ver:paral}
  B_1 \parallel B_2 \iff \text{ there are } e_1,r_1,r_2, L_1,L_2\in\lines_0 
  \text{ such that } 
  \\
  B_1 = e_1 + r_1 L_1 \land B_2 = e_2 + r_2 L_2 \land L_1 \parallel L_2.
\end{multline}
Then we set
$$
  \VerSpace(k,{\goth A}_0) = \struct{\msub_k(S),\lines^\upcircled{k},\parallel}.
$$
\begin{rem}\label{rem:vero:apls}
  Let\/
  ${\goth A}_0$ be an \aplsX,
  $k>1$,
  and ${\goth A} = \VerSpace(k,{\goth A}_0)$.
  Then $\parallel$ defined by \eqref{def:ver:paral} is an equivalence relation 
  and  each of its equivalence classes covers the point set of $\goth A$.
  If a line $L$ of $\goth A$ is given then
  in each of the leaves through a point $f$ there passes exactly one
  line parallel to $L$ with respect to the parallelism just defined.
  \par
  Since there are $k$ leaves through $f$,
  the relation $\parallel$ {\em is not}
  a parallelism in $\struct{\msub_k(S),\lines^\upcircled{k}}$ and
  therefore ${\goth A} = \VerSpace(k,{\goth A}_0)$ 
  {\em is not a \paplsX}.
\end{rem}
\begin{proof}
  It is evident that $\parallel$ defined by \eqref{def:ver:paral} is an 
  equivalence relation in $\lines^\upcircled{k}$.
  To justify the `negative' part note, firstly, that an equivalence class
  $[e+rL]_\parallel$  is determined uniquely by the line
  $L$ of ${\goth A}_0$, 
  and each leaf is isomorphic to ${\goth A}_0$.
\end{proof}

If ${\goth M}_0$ is a partial linear space on 
$\pointnumb_0$ points and $\linenumb_0$ lines with the corresponding
point and line ranks $\pointrank_0$ and $\linerank_0$, then
the parameters of ${\goth M} = \VerSpace(k,{\goth M}_0)$ are as follows:
  $\pointnumb_{\goth M}  =  \binom{\pointnumb_0 + k - 1}{k}$,
  $\pointrank_{\goth M}  =  k\cdot \pointrank_0$,
  $\linerank_{\goth M}   =  \linerank_0$, and
  $\linenumb_{\goth M}   =  \binom{\pointnumb_0 + k - 1}{k-1} \cdot \linenumb_0$.
%
\begin{thm}\label{thm:noparINver}
  Let\/ $k > 1$.
  There is no finite \aplsX\ ${\goth A}_0$ with fixed size of directions of
  its lines such that ${\goth A} = \VerSpace(k,{\goth A}_0)$
  admits a parallelism, the directions of lines of\/ $\goth A$ have a constant size,
  and each leaf of\/ $\goth A$ is an {\em affine subspace} of\/ $\goth A$
  (i.e.\ its each of its leaves is closed under this parallelism).
\end{thm}
\begin{proof}
  Suppose, to the contrary, that there are such $\goth A$ and ${\goth A}_0$;
  let $n$ be the number of points of ${\goth A}_0$.
  Denote by
  $\dirrank_0 = \dirrank_{{\goth A}_0}$
  and $\dirrank = \dirrank_{\goth A}$
  the corresponding sizes of directions.
  Let $\linerank$ be the line size of ${\goth A}_0$,
  and thus of $\goth A$ as well.
  When we observe that a
  direction must cover the point set of $\goth A$ we get
  \begin{ctext}
    $\dirrank_0\cdot\linerank = n$ and
    $\dirrank\cdot \linerank = \binom{n+k-1}{k}$.
  \end{ctext}
  Since no two leaves of $\goth A$ have a line in common,
  a direction of $\goth A$ is a union of directions considered in each of the
  leaves, and each of these leaves is isomorphic to ${\goth A}_0$. 
  This gives us the relation
  %
    $ \dirrank =  \text{ the number of leaves }\cdot \dirrank_0$,
  %
  which yields
  $\binom{n+k-1}{k} = n \binom{n+k-1}{k-1}$,
  so $k=1$.
\end{proof}
In particular, \ref{thm:noparINver} yields a rather strange result: the structure
$\VerSpace(k,{\afgeo(m,q)})$ can be covered by a family of subspaces each one isomorphic to
the affine space $\afgeo(m,q)$ and therefore a natural parallelism intrinsically definable
exists in each of the covering subspaces:
\begin{rem}\label{rem:parincov}
  Let ${\goth A}_0$ be an affine space and 
  let $L_1,L_2$ be lines of\/ $\VerSpace(k,{\goth A}_0)$.
  Then 
  $L_1 \veblparal L_2$ iff there are lines $l_1,l_2$ parallel in ${\goth A}_0$ such that
  and $L_i = e + (k- |e|) l_i$ for some $e$ and $i=1,2$.
\end{rem}
\noindent
The proof of \ref{rem:parincov} is immediate after \eqref{def:veblparal} and \ref{fct:beztroj}.
The union of the parallelisms in leaves is a partial parallelism
and there is no `global' parallelism, that agree in a natural way
with the parallelisms on leaves, definable
on the whole line set of $\VerSpace(k,{\afgeo(m,q)})$.

However, theorem \ref{thm:noparINver} does not mean that $\VerSpace(k,{\goth A}_0)$
does not admit {\em some} parallelism, but this problem is not adressed to this paper.



There are more open problems, that have arised through this paper.
We believe they are interesting on its own right and worth to be considered.
Let us formulate the main of them.
%
%

%
\begin{prob}
The claim of \ref{thm:verred2ver} seems valid also for Veronese spaces associated with projective spaces,
where a quasi-correlation $\varkappa$ is degenerate, since $\dim(\Rad(\varkappa))=1$.
Does \ref{thm:verred2ver} remain true when $\varkappa$ is degenerate and $\dim(\Rad(\varkappa))>1$?
  \par
\end{prob}

\begin{prob}\def\fixpol{\mbox{\boldmath$\goth Q$}}
  If $\varkappa = \varpi$ then the point set
  of $\VerSpace(2,\fixproj)$ and of $\VerSpace(2,\fixpol)$
  (we write ${\fixpol} = \PolarSpace(\varpi)$) coincide.
  The distinction concerns lines only.
  How to characterize $\VerSpace(2,\fixproj)\setminus\hipa$ and $\VerSpace(2,\fixpol)\setminus\hipa$?
  \par
\end{prob}

\begin{prob}
  How to characterize the geometry of the \paplsX\ 
  $\struct{S',\lines',\veblparal}$, where 
  $\struct{S',\lines',\parallel_{\hipa}} = {\fixaf}$
  i.e.\ of the reduct $\VerSpace(2,{\fixproj})$
  equipped with the parallelism imitating the affine one?
\end{prob}

\begin{flushleft}
  \noindent\small
  K. Petelczyc, K. Pra\.zmowski, M. Pra\.zmowska, M. \.Zynel\\
  Institute of Mathematics, University of Bia{\l}ystok,\\
  Akademicka 2, 15-267 Bia{\l}ystok, Poland\\
  \verb+kryzpet@math.uwb.edu.pl+,
  \verb+krzypraz@math.uwb.edu.pl+,
  \verb+malgpraz@math.uwb.edu.pl+,
  \verb+mariusz@math.uwb.edu.pl+
\end{flushleft}

\end{document}